\documentclass{article}
\usepackage{amsthm}
\usepackage{amsfonts,amssymb,amsmath}
\usepackage{enumitem}
\usepackage{titlesec}
\usepackage{fouriernc}
\usepackage[T1]{fontenc}
\usepackage{chngcntr}
\counterwithin{figure}{section}

\newlist{gcases}{enumerate}{1}
\setlist[gcases,1]{
  label={{\it Case}~{\it \Alph*}.},
  topsep=0ex,
  leftmargin=0in,
  labelsep=.1in,
  itemindent=.7in,
  itemsep=0ex
}
\newlist{tenumerate}{enumerate}{1}
\setlist[tenumerate,1]{
  label={(\arabic*)},
  topsep=0ex,
  leftmargin=.3in,
  labelsep=.1in,
  itemindent=0in,
  itemsep=0ex
}

\titleformat{\section}
  {\normalfont\bfseries}   
  {}                             
  {0pt}                          
  {Section \thesection\quad}    

\titleformat{name=\section,numberless}
  {\normalfont\bfseries}
  {}
  {0pt}
  {}

\newlength{\tabwidth}
\newlength{\tabheight}
\newlength{\tabrule}
\newlength{\tabwidthx}
\newlength{\tabheightx}

\def\gentabbox#1#2#3#4{\vbox to \tabheight{\setlength{\tabrule}{#3}%
  \setlength{\tabwidthx}{#1\tabwidth}\addtolength{\tabwidthx}{\tabrule}%

\setlength{\tabheightx}{#2\tabheight}\addtolength{\tabheightx}{-\tabheight}%
  \hbox to #1\tabwidth{%
 \hspace{-0.5\tabrule}\rule{\tabrule}{#2\tabheight}\hspace{-\tabrule}%
    \vbox to #2\tabheight{\hsize=\tabwidthx%
      \vspace{-0.5\tabrule}\hrule width\tabwidthx height\tabrule%
      \vspace{-0.5\tabrule}\vfil%
      \hbox to \tabwidthx{\hss#4\hss}%
        \vfil\vspace{-0.5\tabrule}%
      \hrule width\tabwidthx height\tabrule\vspace{-0.5\tabrule}}%
 \hspace{-\tabrule}\rule{\tabrule}{#2\tabheight}\hspace{-0.5\tabrule}}%
  \vspace{-\tabheightx}}}
\def\genblankbox#1#2{\vbox to \tabheight{\vfil\hbox to
#1\tabwidth{\hfil}}}
\def\tabbox#1#2#3{\gentabbox{#1}{#2}{0.4pt}{\strut #3}}

\catcode`\:=13 \catcode`\.=13 \catcode`\;=13
\catcode`\>=13 \catcode`\^=13
\def:#1\\{\hbox{$#1$}}
\def.#1{\tabbox{1}{1}{$#1$}}
\def>#1{\tabbox{2}{1}{$#1$}}
\def^#1{\tabbox{1}{2}{$#1$}}
\def;{\genblankbox{1}{1}\relax}
\catcode`\:=12 \catcode`\.=12 \catcode`\;=12
\catcode`\>=12 \catcode`\^=7

\newenvironment{tableau}{\bgroup\catcode`\:=13 \catcode`\.=13
  \catcode`\;=13 \catcode`\>=13 \catcode`\^=13
  \setlength{\tabheight}{3ex}\setlength{\tabwidth}{3ex}%
  \def\b##1##2##3{\gentabbox{##1}{##2}{1.2pt}{\vbox{##3}}}%
  \def\n##1##2##3{\gentabbox{##1}{##2}{0.4pt}{\vbox{##3}}}%
  \vbox\bgroup\offinterlineskip}{\egroup\egroup}
\newtheoremstyle{mytheoremstyle} 
    {\topsep}                    
    {\topsep}                    
    {}                   
    {}                           
    {}                   
    {.}                          
    {0.5em}                       
    {\thmnumber{#2.} \thmname{#1}}  

\newtheoremstyle{myremarkstyle} 
    {\topsep}                    
    {\topsep}                    
    {}                   
    {}                           
    {}                   
    {.}                          
    {0.5em}                       
    {\thmname{#1}}  

\theoremstyle{mytheoremstyle}

\newtheorem{theorem}{THEOREM}[section]

\newtheorem{proposition}[theorem]{PROPOSITION}
\newtheorem{definition}[theorem]{DEFINITION}

\theoremstyle{myremarkstyle}
\newtheorem{remark*}{REMARK}

\newcommand\T{\mathbf{T}}

\newcommand\oT{\overline{\mathbf{T}}}

\begin{document}

\noindent
{\bf \large Annihilators and associated varieties of Harish-Chandra modules for $Sp(p,q)$}

\vspace{.3in}
\noindent
WILLIAM MCGOVERN \\
\vspace {.1in}
{\it \small Department of Mathematics, Box 354350, University of Washington, Seattle, WA 98195 }

\section*{Introduction}
The purpose of this paper is to extend the recipes of \cite{G93HC} from the group $SU(p,q)$ to the group Sp$(p,q)$; we will compute annihilators and associated varieties of simple Harish-Chandra modules for the latter group.  We will appeal to the classification of primitive ideals in enveloping algebras of types $B$ and
 $C$ in \cite{G93} via their generalized $\tau$-invariants; thus (inspired by \cite{G93HC}) we will define a map $H$ taking (parameters of) simple Harish-Chandra modules $X$ of trivial infinitesimal character to pairs $(\T_1,\overline{\T}_2)$ where $\T_1$ is a domino tableau and
  $\overline{\T}_2$ an equivalence class of signed tableaux of the same shape as $\T_1$.  Then $\T_1$ will parametrize the annihilator of $X$ (via the classification of primitive ideals in \cite{G93}) while any representative of  $\overline{\T}_2$, suitably normalized, parametrizes its associated variety (via the classification of nilpotent orbits in Lie Sp$(p,q)$ in \cite[9.3.5]{CM93}).  The proof of these properties will rest primarily on the commutativity of the maps $H$ with both
$\tau$-invariants and wall-crossing operators $T_{\alpha\beta}$, defining the latter operators as in  \cite{V79}.  We will parametrize our modules $X$ via signed involutions of signature $(p,q)$ and construct the tableaux from the involutions.

\section{Cartan subgroups and Weyl groups}

For $G=\,\,$Sp$(p,q), n=p+q,p\le q$ we set $\mathfrak g_0 =\,\,$Lie $G$ and we let $\mathfrak g$ be its complexification.  Let $\theta$ be the usual Cartan involution of  $G$ or $\mathfrak g$ and let $\mathfrak k + \mathfrak p$ be the corresponding Cartan decomposition.  Denote by $K$ the subgroup of the complexification of $G$ corresponding to $\mathfrak k$.  Let $H_0$ be a compact Cartan subgroup of $G$ with complexified Lie algebra $\mathfrak h$.   As a choice of simple roots in $\mathfrak g$ relative to $\mathfrak h$ we take $2e_1,e_2-e_1,\ldots,e_{n-1}-e_{n-2},e_n-e_{n-1}$, following \cite{G93}.

There are $p+1$ conjugacy classes of Cartan subgroups of $G$. If we take $H_0$ to be a compact Cartan subgroup and define $H_i$ inductively for $i>0$ as the Cayley transform of $H_{i-1}$ through $e_{p-i+1}-e_{p+i}$ for $1\le i\le p$, then the  $H_i$ furnish a complete set of representatives for the conjugacy classes of Cartan subgroups of $G$.  The real Weyl group $W(H_i)$ of the $i$th  Cartan subgroup $H_i$ is isomorphic to $W_{p-i}\ltimes  S_2^{p-i}\times 
W_i \times W_{q-p+i}$, where $W_r,S_j$ respectively denote the hyperoctahedral group of rank $r$ and the symmetric group on $j$ letters; here $W(H_i)$ embeds into the complex Weyl group $W$ of $\mathfrak g$ (relative to $\mathfrak h$) by permuting, interchanging, and changing the signs of the first $p-i$ pairs of coordinates and then permuting and changing the signs of the next $i$ and then the next $q-p+i$ coordinates \cite{K02,McG98}.  The subgroups $H_i$ are all connected and there is a single block of simple Harish-Chandra modules for $G$ with trivial infinitesimal character.  

\section{The $\mathcal D$ set and Cartan involutions}
Using Vogan's classification of simple Harish-Chandra modules with trivial infinitesimal character by $\mathbb Z/2\mathbb Z$-data \cite{V81}, we parametrize such modules for these groups combinatorially, as follows.   Define $\mathcal S_n$, the set of signed involutions on $n$ letters, to be the set of all sets $\{s_1,\ldots,s_m\}$, where each $s_i$ takes one of the forms $(i,+),(i,-),
(i,j)^+,(i,j)^-$, where $i,j$ lie between 1 and $n$, the pairs $(i,j)$ are ordered with $i<j$, and each index $i$ between 1 and $n$ appears in a unique pair of exactly one of the above types.  We say that $\sigma\in\mathcal S_n$ has signature $(p,q)$ or lies in $\mathcal S_{n,p}$, if the total number of singletons $(i,+)$ and pairs
 $(i,j)^+$ or $(i,j)^-$ in $\sigma$ is $p$.  Let $\mathcal I_n$ denote the set of all involutions in the complex Weyl group $W_n$,  Identify an element $\iota\in\mathcal I_n$ with an element $\sigma\in\mathcal S_n$ by decreeing that $(j,+)\in\sigma$ if $\iota(j) = j, (j,-)\in\sigma$ if $\iota(j) = -j, (i,j)^+\in\sigma$ if the (positive) indices $i,j$ are flipped by $\iota, (i,j)^-\in\sigma$ if instead the indices $i,-j$ are flipped by $\iota$.   Then $W_n$ acts on $\mathcal I_n$ by conjugation and we may transfer this action to $\mathcal S_n$ via the above identification. The set of involutions flipping $p-r$ pairs of indices in $\{\pm1,\ldots,\pm(p+q)\}$ is stable under conjugation and has cardinaility equal to the index of $W(H_{p-r})$, whence we may take $\mathcal S_{n,p}$ as a parametrizing set $\mathcal D$ for the set of simple Harish-Chandra modules for Sp$(p,q)$ with trivial infinitesimal character.  More generally, as needed for inductive arguments below, we define for any subset $M$ of $\{1,\ldots,n\}$ the set $\mathcal S_{M,p}$ in the same way as $\mathcal S_{n,p}$, replacing the numbers 
$1,\ldots,n$ by the numbers in $M$.

The Cartan involution $\theta$ corresponding to any $\sigma\in\mathcal S_{n,p}$ fixes a unit coordinate vector $e_i$ whenever $(i,+)\in\sigma$ or $(i,-)\in\sigma$. If $(i,j)^+\in\sigma$, then
 $\theta$ flips the vectors $e_i$ and $e_j$, while if $(i,j)^-\in\sigma$, then $\theta$ flips $e_i$ and 
 $-e_j$, unless $j-i=1$, in which case $\theta$ sends both $e_i$ and $e_j$ to their negatives.  Thus a simple root $e_{i+1}- e_i$ is compact imaginary for $\sigma$ if and only if $(i,\epsilon)$ and
 $(i+1,\epsilon)$ both lie in $\sigma$ for some sign $\epsilon$, while $2e_1$ is compact imaginary if and only if $(1,\epsilon)$ lies in $\sigma$ for some sign $\epsilon$.
  
 \section{The cross action and the Cayley transform}
 
 For an element $\sigma$ of $\mathcal S_{n,p}$ and a pair of indices $i,j$ between 1 and $n$, we define In$(i,j,\sigma)$ as in \cite[Definition 1.9.1]{G93HC}, interchanging the unique occurrences of $i$ and $j$ in $\sigma$ and leaving all others unchanged (so that $\sigma$ is unchanged if $(i,j)^+$ or $(i,j)^-$ occurs in it).   We define SC$(1,\sigma)$ to be $\sigma$ with the pair $(1,i)^ \epsilon$ in it replaced by $(1,i)^{-\epsilon},-\epsilon$ the opposite sign to $\epsilon$, if there is such a pair in $\sigma$; otherwise SC$(\sigma)=\sigma$.  We define In$(1,2,\sigma)'$ to be $\sigma$ with the pairs $(1,i)^\epsilon,(2,j)^{{\epsilon}'}$ replaced by $(1,j)^{{-{\epsilon}'}},(2,i)^{-\epsilon}$ if such pairs occur in $\sigma$; otherwise we set In$(1,2,\sigma)'=\,$In$(1,2,\sigma)$.
  
 \begin{proposition}\label{proposition:crossaction}
 Let $s$ be the simple root $e_i-e_{i-1},\sigma\in\mathcal S_{n,p}$.  Then $s\times \sigma$, the cross action of $s$ on the parameter $\sigma$, is given by In$(i-1,i,\sigma)$.  For $t = 2e_1$ and $\sigma\in\mathcal S_{n,p}$ we have $t\times\sigma =\, $SC$(1,\sigma)$.
\end{proposition}
 
 \begin{proof} This may be computed directly, along the lines of \cite[\S\S 1.8,9]{G93HC}.
 \end{proof}
 
 As in \cite{G93HC}, we will also need to compute Cayley transforms of parameters $\sigma$ through simple roots.  It suffices to consider the simple root $s=e_{i+1} - e_i$.  Define
  $c^i(\sigma)$ to be $\sigma$ with the pairs $(i,\epsilon),(i+1,-\epsilon)$ replaced by $(i,i+1)^\epsilon$ for $\sigma\in\mathcal S_{n,p}$, if such pairs occur in $\sigma$; otherwise $c^i(\sigma)$ is undefined.  Thus $c^i(\sigma)$ is defined if and only if $e_{i+1}-e_i$ is an imaginary noncompact root for $\sigma$.  In a similar way, define $c_i(\sigma)$ if $e_{i+1}-e_i$ is real for $\sigma$ to be the i through nvolution obtained from $\sigma$ by replacing $(i,i+1)^\epsilon$ by $(i,\epsilon),(i+1,-\epsilon)$.  
  
\begin{proposition}\label{proposition:transform}
If $e_{i+1} - e_i$ is imaginary noncompact for a parameter $\sigma$ in $\mathcal S_{n,p}$, then the Cayley transform of $\sigma$ through this root is given by $c^i(\sigma)$.  If this root is real for $\sigma$, then the Cayley transform through it is given by $c_i(\sigma)$.
 \end{proposition}
 
 \begin{proof} Again this follows from a direct computation, along the lines of \cite[1.12]{G93HC}.
 \end{proof}
 
 \section{$\tau$-invariants and wall-crossing operators}
 
 We define the $\tau$-invariant $\tau(\sigma)$ of a parameter $\sigma$in $\mathcal S_{n,p}$ in the same way as in \cite[1.13]{G93HC}: it consists of the simple roots that are either real, compact imaginary, or complex and sent to negative roots by the Cartan involution $\theta$.  We extend this definition to $\mathcal S_{M,p}$ as in \cite[1.15]{G93HC}.  Similarly, if $\alpha,\beta$ are nonorthogonal simple roots of the same length, then we define the wall-crossing operator $T_{\alpha\beta}$ on $\mathcal S_{n,p}$ as in \cite[1.14]{G93HC} and \cite{V79}.  It is single-valued.  If $\alpha,\beta$ are nonorthogonal but have different lengths we define
  $T_{\alpha\beta}$ on $\mathcal S_{n,p}$ in the same way; in this setting
 it takes either one or two values.  In both cases $T_{\alpha\beta}$ sends parameters with $\alpha$ not in the $\tau$-invariant but $\beta$ in it to parameters with the opposite property.  In more detail, if a parameter $\sigma$ includes $(1,2)^+$, then the effect of $T_{\alpha\beta}$ on it is to replace $(1,2)^+$ by either $1^+, 2^-$ or $1^-,2^+$, and vice versa; otherwise we interchange $1$ and $2$ in $\sigma$ if at least one of them is paired with another index but they are not paired with each other, or we change the sign attached to the pair $(1,i)^\epsilon$ in $\sigma$, whichever (or both) of these operations has the desired effect on the $\tau$-invariant of $\sigma$.
 
  \section{The algorithm}
 
 We now describe the algorithm we will use to compute for a given $\sigma\in\mathcal S_{n,p}$ the annihilator and associated variety of the corresponding Harish-Chandra module for Sp$(p,q)$.  We  will attach an ordered pair $(\T_1,\overline{\T}_2)$ to $\sigma$, where $\T_1$ is a domino tableau in the sense of \cite{G90} and $\overline{\T}_2$ is an equivalence class of signed tableaux of signature $(2p,2q)$ and the same shape as $\T_1$.  The shape of $\T_1$ will be a doubled partition of $2(p+q)$; that is, a partition of $2(p+q)$ in which the parts occur in equal pairs.  Any tableau in $\overline{\T}_2$ will thus also have rows occurring in pairs, called double rows, of equal length; moreover the two rows in a double row will begin with the same sign.  The equivalence relation defining $\overline{\T}_2$ will be that we can change all signs in any pair of double rows of the same even length, or in any pair $D_1,D_2$ of double rows of different even lengths whenever there is an open cycle of $\T_1$ in the sense of \cite[3.1.1]{G93} (so not including its smallest domino) with its hole in one of $D_1,D_2$ and its corner in the other.  Here we allow the second double row $D_2$ to have length 0, so that, for example, if 
 \vskip .1in
 $$
 \raisebox{2ex}{$\T_1=$\;}
 \begin{small}
 \begin{tableau}
 :>1\\
 :>2\\
 \end{tableau}
 \end{small}
 $$
 \vskip .1in
 \noindent then the two signed tableaux of the same shape (one with its rows beginning with $+$, the other with $-$) both lie in $\overline{\T}_2$, since if $\T_1$ is moved through the open cycle of its 2-domino, then that domino is given a clockwise quarter-turn, so that it now occupies one square of a double row that was empty in $\T_1$.  On the other hand, if $\T_1$ is replaced by its transpose, then the two signed tableaux of this shape lie in separate classes $\overline{\T}_2$, since in that case the open cycle of the 2-domino does not intersect the empty double row.  In addition to this equivalence relation, we decree as usual for signed tableaux that any two of them are identified whenever one can be obtained from the other by interchanging pairs of double rows of the same length.  The signature of $\overline{\T}_2$ (i.e. the number of $+$ signs in it) will be $2p$; note that this is an invariant of $\overline{\T}_2$.  To construct 
 $\T_1$ and $\overline{\T}_2$ we follow a similar recipe to \cite[Chap.\ 3]{G93HC}, replacing the tableaux occurring there with domino tableaux and using insertion and bumping for domino tableaux as in \cite{G90}.
 
 \begin{definition}\label{definition:addsingleton}
 Let $\sigma\in\mathcal S_{n,p}$.  Order the elements of $\sigma$ by increasing size of their largest numbers.  We construct the pair $H(\sigma) = (\T_1,\overline{\T}_2)$ attached to $\sigma$ inductively,  starting from a pair of empty tableaux.  At each step we insert the next element $(i,\epsilon)$ or
 $ (i,j)^\epsilon$ into the current pair of tableaux.  Assume first that the next element of $\sigma$ is $(i,\epsilon)$ (with $\epsilon$ a sign) and choose any representative $\T_2$ of $\overline{\T}_2$. 
 
 \begin{tenumerate}
 
 \item If the first double row of $\T_2$ ends in $-\epsilon$, then add $\epsilon$ to the end of both of its rows and add a vertical domino labelled $i$ to the end of the first double row of
  $\T_1$.   
 
 \item If not and if the first double row of $\T_2$ has (rows of) even length, then we look first for a lower double row of $\T_2$ with the same length ending in $-\epsilon$; if there is a such a double row, we interchange it with the first double row in $\T_2$ and then proceed as above.  Otherwise we start over, trying to insert $(i,\epsilon)$ into the highest double row of $\T_2$ strictly shorter than its first double row.  (In the end, we may have to insert a domino labelled $i$ into a new double row of $\T_1$, using $\epsilon$ for both signs in the new double row of $\T_2$.)
 
 \item If not and the first (or first available) double row of $\T_2$ has odd length but there is more than one double row of this length, but none ending in $-\epsilon$, then we change all signs in the first two double rows of $\T_2$ of this length and then proceed as in the previous case.
  
 \item Otherwise the highest available double row $R$ in $\T_2$ has even length, ends in
  $\epsilon$, and is the only double row of this length.  In this case we look at the domino in $\T_1$ occupying the last square in the lower row of $R$.  If we move $\T_1$ through the open cycle of this domino, we find that its shape changes by removing this square and adding a square either at the end of the higher row of some double row $R'$ of $\T_1$ or else in a new row, not in $\T_1$.  If it lies in a new row, then change all signs in $R$ and proceed as above.  If it does not lie in a new row and $R'\ne R$, then change the signs of $\T_2$ in both $R$ and $R'$ and proceed as above (again not actually moving $\T_1$ through the cycle).  Finally, if $R=R'$, then move $\T_1$ through the open cycle, place a new horizontal domino labelled $i$ at the end of the lower row of $R$ in $\T_1$, and choose the signs in $\T_2$ so that both rows of $R$ now end in $\epsilon$ while all other rows of $\T_2$ have the same signs as before.
 \end{tenumerate}
 \end{definition}

\begin{definition}\label{definition:addpair}
Retain the notation of the previous definition but assume now that the next element of $\sigma$ is
  $(i,j)^\epsilon$.  We begin by inserting a horizontal domino labelled $i$ at the end of the first row of $\T_1$ if $\epsilon=+$, or a vertical domino labelled $i$ at the end of the first column of $\T_1$ if
  $\epsilon=-$, following the procedure of \cite{G90} (and thus bumping dominos with higher labels as needed).  We obtain a new tableau $\T'$, whose shape is obtained that of $\T_1$ by adding a single domino $D$, lying either in some double row $R$ of
   $\T_1$ or else in a new row (in which case $D$ must be horizontal).   Let $\ell$ be the length of $R$ (before $D$ was added).

\begin{tenumerate}
    \item If $D$ is horizontal and $\ell$ is even, then add a domino labelled $j$ to $\T'$ immediately below the position of $D$, in the lower row of $R$.  Choose signs in $\T_2$ so that both rows of $R$ now end in a different sign than they did before; leave all other signs the same.
If $D$ lies in a new row, then we have a new double row in $\T_2$, which can begin with either sign; to make a particular choice, we decree that both rows in the new double row begin with $-$.

    \item If $D$ is horizontal and $\ell$ is odd, then $\T'$ does not have special shape in the sense of \cite{G90}, but its shape becomes special if one moves through just one open cycle.  Move through this cycle and choose the signs in $\T_2$ so that $R$ is now a genuine double row and its rows end in a different sign than they did before.  Then $\T_2$ has either two more $+$ signs than before or two more $-$ signs.  Insert a vertical domino labelled $j$ to the first available double row in $\T_1$ strictly below $R$, following the procedure of the previous definition.  The sign attached to $j$ is $-$ if $\T_2$ gained two $+$ signs and is $+$ otherwise.
    
    \item If $D$ is vertical and $\ell$ is even, then $R$ is still a double row; choose signs so that its rows end in the same sign as they did before, leaving all other signs unchanged.  If a new double row was created by inserting the new domino, then its rows can begin with either sign; to make a particular choice, we decree that this sign is $+$.  Now add a new vertical domino $j$ to the first available double row strictly below $R$, as in the previous case, giving it the same sign as in that case.
    
    \item If $D$ is vertical and $\ell$ is odd, then proceed as in the previous case.
\end{tenumerate}
\end{definition}
\vskip .1in

\noindent One can check that either choice of sign made in Definition 5.2(1) or (3) gives rise to equivalent tableaux $\T_2$, so that in the end we get a well-defined equivalence class $\overline{\T}_2$.   To compute the associated variety of the Harish-Chandra module corresponding to $\sigma$, we choose any representative of $\overline{\T}_2$ and normalize it so that all even rows begin with $+$; this is because this variety is the closure of one nilpotent $K$-orbit in $\mathfrak p^*$ \cite[5.2]{T07} and such orbits (via the Kostant-Sekiguchi bijection between them and nilpotent orbits in $\mathfrak g_0$) are parametrized by signed tableaux of signature $(2p,2q)$ in which even rows begin with $+$ \cite[9.3.5]{CM93}.  Later we will show that the map $H$ defines a bijection between $\mathcal S_{n,p}$ and pairs $(\T_1,\overline{\T}_2)$ with $\overline{\T}_2$ of signature $(2p,2q)$
\vskip .2in
\noindent We give two examples.  First let $\sigma = \{(1,2)^+\}\in\mathcal S_{2,1}$.  Then we get 
$$
\raisebox{2ex} {$\T_1=$\;}
\begin{small}
\begin{tableau}
:>1\\
:>2\\
\end{tableau}
\end{small}
$$
\noindent while $\overline{\T}_2$ consists of both signed tableaux of this shape, as noted above.
In the other example, we let $\sigma = \{((1,+),(2,-),(3,4)^+,(5,+)\}\in\mathcal S_{5,3}$.
Then
$$
\raisebox{2ex} {$\T_1=$\:}
\begin{small}
\begin{tableau}
:^1>2>3\\
:;>4>5\\
\end{tableau}
\end{small}
\hspace{.2in}
\raisebox{2ex}{and \hspace{.2in}$\T_2=$\;}
\begin{small}
\begin{tableau}
:^+^-^+^-^+\\
:;;;;;\\
\end{tableau}
\end{small}
$$
\noindent where we denote signed tableaux by tableaux tiled by vertical dominos, each labelled with the (common) sign of each of its squares; note that here $\overline{\T}_2$ consists of a single tableau.

 \section{$\tau$-invariants and $\T_{\alpha\beta}$ on tableaux}
 
 We define the $\tau$-invariant $\tau(\T_1,\overline{\T}_2)$ of a pair $(\T_1,\overline{\T}_2)$, or just of its domino tableau $\T_1$, in the same way as \cite[2.1.9]{G92}; thus for example $2e_1$ lies in the $\tau$-invariant of $\T_1$ in type $C$ if and only if the 1-domino in it is vertical.  If $\alpha,\beta$ are simple roots of the form $e_i-e_{i-1},e_{i+1}-e_i$ for some $i$, then we define $T_{\alpha\beta}$ on a domino tableau $\T_1$ lying in the domain $D_{\alpha\beta}$ of this operator as in \cite[2.1.10]{G92} and extend this to a pair $(\T_1,\overline{\T}_2)$ by making the operator act trivially on $\overline{\T}_2$.   We now define $T_{\alpha\beta}$ in the other cases, following the notation of \cite[2.3.4]{G92}, and as in \cite{G92} defining this operator on pairs rather than single tableaux. 
 
 \begin{definition}\label{definition:Coperator}
 Suppose that $\alpha=2e_1,\beta=e_2 -e_1$.  If $\T_1\in D_{\alpha\beta}$, then either $F_2\subseteq \T_1$ or $\tilde{F}_2\subseteq\T_1$.
 
 \begin{tenumerate}
 \item  In the first case, let $\T_1'$ be obtained from
  $\T_1$ by replacing $F_2 $ by $F_1$.   If the 2-domino of $\T_1'$ lies in an open cycle not including the 1-domino and if the equivalence class $\overline{\T}_2$ breaks up into two classes $\overline{\T}_2',\overline{\T}_2''$ with respect to $\T_1'$, then we set $T_{\alpha\beta}(\T_1,\overline{\T}_2) = ((\T_1',\overline{\T}_2'),(\T_1',\overline{\T}_2''))$.  If the 2-domino lies in a closed cycle $c$, then let $\tilde\T_1'$ be the tableau obtained from $\T_1'$ by moving through $c$ and we set $T_{\alpha\beta}(\T_1,\overline{\T}_2) =((\tilde\T_1',\overline{\T}_2),(\T_1',\overline{\T}_2))$.  Otherwise set $T_{\alpha\beta}(\T_1,\overline{\T}_2) = (\T_1',\overline{\T}_2)$.
  
  \item  If instead $\tilde{F}_2\subseteq\T_1$, then the 2-domino of $\T_1$ lies in a closed cycle $c$, since $\T_1$ has the (special) shape of a doubled partition; if this cycle were open, it would have to be simultaneously an up and down cycle in the sense of \cite[\S3]{G93}, a contradiction.  Let $\tilde{\T}_1$ be obtained from $\T_1$ by moving through $c$ and let $\tilde{\T}_1'$ be obtained from $\tilde{\T}_1$ by replacing $F_2$ by $F_1$.  Then set $T_{\alpha\beta}(\T_1,\overline{\T}_2)$ = $(\tilde{\T}_1',\overline{\T}_2)$.
  
  \item If instead $\alpha=e_2 -e_1,\beta=2e_1$, then define $T_{\alpha\beta}(\T_1)$ for 
  $\T_1\in D_{\alpha\beta}$ as above, interchanging $F_1,F_2$ throughout by $\tilde{F}_1,\tilde{F}_2$.
  \end{tenumerate}
   \end{definition}
  \vskip .2in
  \noindent For example, if $\T_1$ is as in the first example in the last section, so that $\overline{\T}_2$ consists of both signed tableaux of this shape, then $T_{2e_1,e_2-e_1}$ sends $(\T_1,\overline{\T}_2)$ to the pair $((\T',\T_2'),(\T',\T_2'')$, where $\T'$ the transpose of $\T_1$ and $\T_2',\T_2''$ are the two signed tableaux in $\overline{\T}_2$.  Note also that, unlike \cite[2.3.4]{G92}, we must not move through any open cycles, as all of our tableaux must have doubled partition shape.  There are no right domino tableaux and so there is no notion of extended cycle.

 \section{$H$ commutes with $\tau$-invariants}
 
 As in \cite{G93HC}, we prove that our algorithm $H$ computes the annihilators of simple Harish-Chandra modules by showing that it commutes with taking $\tau$-invariants and applying wall-crossing operators.  In this section we deal with $\tau$-invariants.
 
 \begin{proposition}\label{proposition:tau}
 Let $\sigma\in\mathcal S_{n,p}$ and $\alpha$ a simple root for Sp$(p,q)$.  Then
  $\alpha\in\tau(\sigma)$ if and only if $\alpha\in\tau(H(\sigma)$.
   \end{proposition}
 
 \begin{proof} We enumerate all possible ways in which $\alpha$ can lie in $\tau(\sigma)$, or fail to lie in this set, and then check directly that the conclusion holds in each case.
 \vskip .1in
\noindent Suppose first that $\alpha\in\tau(\sigma)$.
 \vskip .1in
 \begin{tenumerate}
 
 \item If $\alpha=e_{i+1}-e_i$ is compact imaginary, then we must have $(i,\epsilon),(i+1,\epsilon)\in
 \sigma$ for some sign $\epsilon$.  The $i$-domino starts out vertical and in the first double row of
  $\T_1$; eventually either the $i$- and $(i+1)$-dominos wind up horizontal with the first on top of the second, or the $(i+1)$-domino is added vertically to a lower double row.  In both cases the $i$-domino winds up above the $(i+1)$-domino, as desired.  If $\alpha=2e_1$ is compact imaginary, then $(1,\epsilon)\in\sigma$ for some sign $\epsilon$, so that the 1-domino is vertical in $\T_1$.  
  
  \item If $\alpha=e_{i+1}-e_i$ is real, then we must have $(i,i+1)^+\in\sigma$.  It is clear from the algorithm that the $i$-domino winds up below the $(i+1)$-domino.
  
  \item Otherwise $\alpha$ is complex.  If $\alpha=2e_1$, then we must have $(1,j)^-\in\sigma$ for some $j$, and then it is clear that the 1-domino winds up vertical in the first double row of $\T_1$, while the $j$-domino lies below this double row.  
  
  \item We are now reduced to the case where $\alpha=e_{i+1}-e_i, \alpha$ complex.  If\linebreak
   $(i,\epsilon),
  (j,i+1)^{\epsilon'}\in\sigma$ for some $j<i$ for signs $\epsilon,\epsilon'$ and if the $i$-domino is vertical  when adjoined to $\T_1$, then it is added to the end of some double row $R$ such that the double rows above it end in the same sign as $R$ in $\T_2$ (since the $i$-domino was not put into a higher row).  When the $j$-domino is inserted, adding a domino $D$ to the shape of $\T_1$, the additional signs added to $\T_2$, if $D$ is vertical, are both $-\epsilon$, whence the $(i+1)$-domino is now inserted vertically with sign $\epsilon$ and winds up in a row below $R$ (since all higher rows end with the same sign as they did before the $j$-domino was inserted).  If $D$ is horizontal, then it lies in the bottom row of $\T_1$, and once again, the $(i+1)$-domino lies below it.  The argument is similar if instead the $i$-domino ends up horizontal (lying directly below another horizontal domino) when it is adjoined to $\T_1$.
  
  \item If $(i+1,\epsilon),(i,k)^+\in\sigma$ for some $k>i+1$, then the $(i+1)$-domino is vertical at the end of some double row of $\T_1$; the
  $i$-domino is adjoined horizontally either to this double row or a higher one, and if to this double row bumps the $(i+1)$-domino so that it lies below the $i$-domino, as desired.  
  
  \item If $(j,i+1)^\epsilon,(i,k)^+\in\sigma$ and $j<i<k$, then the $i$-domino is added to the first row of $\T_1$, while the 
  $(i+1)$-domino in all cases lies below this row.  A similar argument applies if instead $(j,i)^\epsilon,
  (i+1,k)^-\in\sigma$.
  
  \item If $(j_2,i)^+,(j_1,i+1)^\epsilon\in\sigma$ and $j_1<j_2<i$, then adding the $j_1$-domino bumps the dominos that were previously bumped by adding the $j_2$-domino, together with at least one domino in the double row of the $i$-domino, so that the $(i+1)$-domino winds up below the $i$-domino. 
    
  \item If $(j_1,i)^\epsilon,(j_2,i+1)^-\in\sigma$ and $j_1<j_2<i$, then since the 
  $j_2$-domino is inserted vertically into the first column of $\T_1$, the $(i+1)$-domino is again 
  forced to lie below the $i$-domino, as in the previous case.
  
  \item if $(j_2,i)^+,(j_1,i+1)^-\in\sigma$ and $j_1<j_2$, then again the $(i+1)$-domino winds up below the
  $i$-domino, similarly to the previous case.
  
  \item if $(i+1,k_1)^\epsilon,(i_1,k_2)^+\in\sigma$ and $i+1<k_1<k_2$, then either the
  $i$-domino bumps the $(i+1)$-domino, or else the $i$-domino winds up horizontal in the first double row, while the $(i+1)$-domino winds up vertical in the first column, in either case lying below the
  $i$-domino.
  
  \item if $(i,k_1)^-,(i+1,k_2)^-\in\sigma$ and  $i+1<k_1<k_2$, then the $(i+1)$-domino
  is inserted below the $i$-domino.
  
  \item if $(i,k_1)^+,(i+1,k_2)^-\in\sigma$ and $i+1<k_1<k_2$, then as above the $i$-domino winds up in the first double row while the $(i+1)$-domino winds up in the first column and lies below the former.
  
  \end{tenumerate}
  \vskip .1in
  \noindent This exhausts all cases where $\alpha\in\tau(\sigma)$.  Now suppose the contrary.  The cases where $\alpha=2e_1$ are easily dealt with, so assume that
   $\alpha=e_{i+1}-e_i$.
   \vskip .1in
  
  \begin{tenumerate}
  
  \item If $\alpha$ is noncompact imaginary, so that $(i,\epsilon),(i+1,-\epsilon)\in\sigma$, then the
  $(i+1)$-domino is inserted either next to the $i$-domino or in a higher row.
  
  \item If $(j,i)^-,(i+1,\epsilon)\in\sigma$ and $j<i$, then either the $(i+1)$-domino is
  added vertically at the end of the first double row, or there is at least one double row of odd length ending in $-\epsilon$ in $\T_1$, whence the $(i+1)$-domino is added to a row not below the
  $i$-domino.
  
  \item If $(i,\epsilon),(i+1,k)'+\in\sigma$ and $+1i<k$, then the $(i+1)$-domino is inserted
  horizontally into the first row and cannot lie below the $i$-domino.
  
  \item If $(i+1,\epsilon),(i,k)^-\in\sigma$ and $i<k$ then the $(i+1)$-domino is inserted vertically into the first column and cannot lie below the $i$-domino.
  
  \item If $(j,i)^\epsilon,(i+1,m)^+\in\sigma$ and $j<i$, then the $(i+1)$-domino is inserted horizontally into the first row and cannot lie below the $i$-domino.
  
  \item If $(j_1,i)^\epsilon,(j_2,i+1)^+\in\sigma$ and $j_1<j_2<i$, then either the $(i+1)$-domino is inserted vertically into a double row no further down than the one in which the $i$-domino appears, or the $i$-domino is in the first double row and the $(i+1)$-domino is also inserted into this double row, not directly below the former.  In both cases the $(i+1)$-domino is not below the $i$-domino.
  
  \item If $(j_2,i)^-,(j_1,i+1)^+\in\sigma$ and $j_1>j_2<i$, then the $i$-domino is in the lowest double row of $\T_1$ and the $(i+1)$-domino cannot lie below it.
  
  \item If $(j_2,i)^-,(j_1,i+1)^-\in\sigma$ and $j_1<j_2<i$, then the $i$-domino is in the lowest double row of $\T_1$ and $(i+1)$-domino is not below this double row.
  
  \item If $(i,k_1)^\epsilon,(i+1,k_2)^+\in\sigma$ and $i+1<k_1<k_2$, then the $(i+1)$-domino is inserted in the first double row of $\T_1$ and does not lie below the $i$-domino. 
  
  \item If $(i+1,k_1)^-,(i,k_2)^-\in\sigma$ and $i+1<k_1<k_2$, then the $i$-domino bumps the
   $(i+1)$-domino vertically and does not wind up below the latter.
  
  \item If $(i+1,k_1)^+,(i,k_2)^-\in\sigma, i+1<k_1<k_2$, then the $(i+1)$-domino is inserted into the first double row and cannot lie below the $i$-domino.
  
  \end{tenumerate}
  \vskip .1in
  \noindent This exhausts all cases and concludes the proof.
  
  \end{proof}
  
  \section{$H$ commutes with $T_{\alpha\beta}$}.
  
  \noindent Again following \cite{G93HC}, we complete our program of showing that the map $H$ computes annihilators by showing that it commutes with wall-crossing operators.
  
  \begin{proposition}\label{proposition:wall-cross} Let $\alpha,\beta$ be nonorthogonal simple roots and let $\sigma\in\mathcal S_{n,p}$  lie in the domain of the operator $T_{\alpha\beta}$.  Then $H(T_{\alpha\beta}(\sigma)) = T_{\alpha\beta}(H(\sigma)$.
  \end{proposition}
  
\begin{proof} As in \cite{G93HC} we enumerate all ways in which $\sigma$ can lie in the domain of 
$T_{\alpha\beta}$ and check that the conclusion holds in all cases.  Let $\T_1,\T_2$ respectively denote the domino tableau and a representative of the class of signed tableaux attached to $\sigma$ by the algorithm.
  
Suppose first that $\{\alpha,\beta\}=\{2e_1,e_2-e_1\}$. If $(1,2)^+\in\sigma$, then $F_2\subseteq\T_1$ and $T_{\alpha\beta}(\sigma)$ consists of the two involutions $\sigma_1,\sigma_2$ obtained from $\sigma$ by replacing $(1,2)^+$ by $(1,+),(2,-)$ and $(1,-),(2,+)$ in turn.  Clearly the domino tableaux attached to $\sigma_1,\sigma_2$ are both equal to the tableau $\T_1'$ defined in Definition 6.1 (1).  The signed tableaux attached to these involutions at the second step of the algorithm are not equivalent at that step, whence by the algorithm they remain
  inequivalent at its end.  Hence $H(\sigma_1)\ne H(\sigma_2)$, as desired.  The cases where $(1,+),(2,-)\in\sigma$ or $(1,-),(2,+)\in\sigma$ are similar.  Finally, in the case where at least one of the indices 1 and 2 is paired with another index but 1 and 2 are not paired with each other, one clearly moves the 1- and 2-dominos in $\T_1$ in the desired fashion, whence one can check that if any other dominos move, they are the ones in the closed cycle containing the 2-domino of $\T_1$ and in fact the two domino tableaux produced are those specified by Definition 6.1 (cf.\ \cite[2.2.9,2.3.7]{G92}). If the 2-domino of $\T_1$ does not lie in a closed cycle, then only one domino tableau is produced, which again agrees with that given by this definition.
  
Henceforth we assume that $\alpha=e_i-e_{i-1},\beta=e_{i+1}-e_i$ for some $i\ge2$.  Set $\sigma' = T_{\alpha\beta}(\sigma)$ and let $\T_1',\T_2'$ respectively denote the domino tableau and a representative of the class of signed tableaux attached to $\sigma'$ by the algorithm.  The cases in our discussion below are parallel to the corresponding cases in the proof of \cite[Proposition 4.2.1]{G93HC}.  That proof shows that the desired result holds whenever none of the indices $i-1,i,i+1$ occurs in an ordered pair $(a,b)^-$ in $\sigma$ and $T_{\alpha\beta}$ does not act on $\T_1$ by an $F$-type interchange in the sense of \cite{G92}, using \cite[2.1.20,2.1.21]{G92} in place of the results in Section 2.5 of \cite{G93HC}:  in all cases either the $(i-1)$- and $i$- or $i$- and $(i+1)$-dominos are interchanged in $\T_1$, whichever of these has the desired effect on $\tau$-invariants.  Apart from this one changes signs and moves through open cycles in the same way in the constructions of $\T_1,\T_1'$ and $\T_2,\T_2'$, so that $\T_2'$ is equivalent to $\T_2$, as desired.  If an ordered pair $(a,b)^-$ involving one of the indices $i-1,i,i+1$ does occur in $\sigma$, then one checks directly that $\T_1'$ is obtained from $\T_1$ by either interchanging the $(i-1)$- and $i$- or $i$- and $(i+1)$-dominos and we may take $\T_2'=\T_2$, as desired; note that ordered pairs $(a,b)^-$ have no analogue in \cite{G93HC}..  It only remains to show that the desired result holds whenever $\T_{\alpha\beta}$ acts on $\T_1$ by an $F$-type interchange (again, there is no analogue of such an interchange in \cite{G93HC}).   In each case below, we indicate how many subcases involve an $F$-type interchange; then the result follows by a direct calculation in each such subcase.  Throughout we denote by $j,j_1,j_2,j_3$ indices less than $i-1$ with $j_1<j_2<j_3$, and similarly by $k,k_1,k_2,k_3$ indices greater than $i+1$ with $k_1<k_2<k_3$.    
\vskip .1in
\begin{tenumerate}

\item Suppose first that $(i-1,\epsilon),(i,-\epsilon),(i+1,-\epsilon)$ all lie in $\sigma$ for some sign $\epsilon$, which for definiteness we take to be $+$.  Then $\sigma'$ is obtained from $\sigma$ by replacing the terms $(i-1,+),(i,-)$ by the single term $(i-1,i)^+$.  Let $\tilde\sigma$ consist of the terms of $\sigma$ involving only indices less than $i-1$ and let $\tilde\T_1,\tilde\T_2$ be the domino and representative of the class of signed tableaux attached to $\tilde\sigma$ by the algorithm.   There are four subcases, according as the top double row of $\tilde\T_2$ has even or odd length and ends with $+$ or $-$, but only one of these has $\T_{\alpha\beta}$ acting on $\T_1$ by an $F$-type interchange.  One checks directly that the conclusion holds in this case.

\item If instead $(i-1,\epsilon),(i,i+1)^+\in\sigma$, so that $(i-1,\epsilon),(i,\epsilon),(i+1,-\epsilon))\in\sigma'$, then again only one subcase out of four has $T_{\alpha\beta}$ acting on $\T_1$ by an $F$-type interchange, and the conclusion holds in that case.

\item If $(j,i-1)^\epsilon,(i,i+1)^+\in\sigma$, so that $(j,i)^\epsilon,(i-1,i+1)^+\in\sigma'$, then no $F$-type interchange ever takes place.

\item If $(i-1,i+1)^\epsilon,(i,k)^+\in\sigma$, so that $(i-1,i)^\epsilon,(i+1,k)^+\in\sigma'$, then no $F$-type interchange ever takes place.

\item If $(j,i-1)^\epsilon,(i,\epsilon'),(i+1,\epsilon')\in\sigma$, so that $(i-1,\epsilon'),(j,i)^\epsilon),
((i+1),\epsilon')\in\sigma'$, then there are eight subcases, depending as in case 1 on the length parity and sign at the end of the top double row, and this time also on whether the $i-1$- and $i+1$-dominos are bumped into the same double row.  Two subcases involve an $F$-type interchange and the desired result holds in both of them.

\item If $(i-1,\epsilon),(i,-\epsilon),(j,i+1)^{\epsilon'}\in\sigma$, so that $(i-1,\epsilon),(j,i)^{\epsilon'},
(i+1,-\epsilon)\in\sigma'$, then no $F$-type interchange takes place.

\item If $(i-1,\epsilon),(i+1,\epsilon),(i,k)^+\in\sigma$, so that $(i-1,\epsilon),(i,\epsilon),(i+1,k)^+\in\sigma'$, then there is one case where an $F$-type interchange occurs and the result holds in that case.

\item If $(i-1,\epsilon),(i+1,-\epsilon),(i,k)^+\in\sigma$, so that $(i,\epsilon),(i+1,-\epsilon),(i-1,k)^+\in\sigma'$, then an $F$-type interchange always occurs and the result holds in all cases.

\item If $(j_1,i-1)^+,(i,\epsilon'),(j_2,i+1)^+\in\sigma$, so that $(i-1,\epsilon),(j_1,i)^+,(j_2,i+1)^+\in\sigma'$, then there is one case where an $F$-type interchange occurs and the result holds in it.

\item If $(j_2,i-1)^+,(i,\epsilon),(j_1,i+1)^+\in\sigma$, so that $(j_2,i-1)^+,(j_1,i)^+,(i+1,\epsilon)^+\in\sigma'$, then an $F$-type interchange never arises.

\item If $(j,i-1)^+,(i+1,\epsilon),(i,k)^+\in\sigma$, so that $(j,i)^+,(i+1,\epsilon),(i-1,k)^+\in\sigma'$, then an $F$-type interchange never arises.

\item If $(i-1,\epsilon),(j,i+1)^+,(i,k)^+\in\sigma$, so that $(i-1,\epsilon),(j,i)^+,(i+1,k)^+\in\sigma'$, then an $F$-type interchange never arises.

\item If $((i+1,\epsilon),(i-1,k_1)^+,(i,k_2)^+\in\sigma$, so that $(i,\epsilon),(i-1,k_1),(i+1,k_2)\in\sigma'$, then there are two subcases where an $F$-type interchange arises and the result holds in both of them.

\item If $(i-1,\epsilon),(i+1,k_1)^+,(i,k_2)^+\in\sigma$, so that $((i,\epsilon),(i+1,k_1^+),(i-1,k_2)^+\in\sigma'$, then there are two subcases where an $F$-type interchange arises and the result holds in both of them.

\item If $((j_2,i-1)^+,(j_3,i)^+,(j_1,i+1)^+\in\sigma$, so that $(j_2,i-1)^+,(j_1,i)^+,(j_3,i+1)^+\in\sigma'$, then no $F$-type interchange occurs.

\item If $(j_1,i-1)^+,(j_3,i)^+,(j_2,i+1)^+\in\sigma$, so that $(j_3,i-1)^+,(j_1,i)^+,(j_2,i+1)\in\sigma'$, then no $F$-type interchange toccurse.

\item If $(j_2,i-1)^+,(j_i,i+1)^+,(i,k)^+\in\sigma$, so that $(j_2,i-1)^+,(j_1,i)^+,(i+1,k)^+\in\sigma'$, then no $F$-type interchange occurs.

\item If $(j_1,i-1)^+,(j_2,i)^+,(i,k)^+\in\sigma$, so that $(j_1,i)^+,(j_2,i+1)^+,(i-1,k)^+\in\sigma'$, then no $F$-type interchange occurs.

\item If $(j,i+1)^+,(i-1,k_1),(i,k_2)^+\in\sigma$, so that $(j,i)^+,(i-1,k_1)^+,(i+1,k_2)^+\in\sigma'$, then no $F$-type interchange occurs.

\item If $(j,i-1)^+,(i+1,k_1)^+,(i,k_2)^+\in\sigma$, so that $(j_i^+,(i+1,k_1)^+,(i-1,k_2)^+\in\sigma'$, then no $F$-type interchange occurs.

\item If $(i+1,k_1)^+,(i-1,k_2)^+(,i,k_3)^+\in\sigma$, so that $(i,k_1)^+,(i-1,k_2)^+,(i+1,k_3)^+\in\sigma'$, then no $F$-type interchange occurs.

\item If $(i-1,k_1)^+,(i+1,k_2)^+,(i,k_3)^+\in\sigma$, so that $((i,k_1)^+,(i+1,k_2)^+,(i-1,k_3)^+\in\sigma'$, then no $F$-type interchange occurs.
 
  \end{tenumerate}
  \vskip .1in
  \noindent This exhausts all cases and concludes the proof.
  
  \end{proof}
 
  \begin{theorem}\label{theorem:annihilators} Let $\sigma\in\mathcal S_{n,p}$.  Then the first coordinate $\T_1$ of $H(\sigma)$ parametrizes the annihilator of the simple Harish-Chandra module corresponding to $\sigma$ via the classification of \cite[Theorem 3.5.11]{G93}.
  \end{theorem}
  
  \begin{proof} Thanks to Propositions 7.1 and 8.1, we know that the primitive ideal $I$ corresponding to $\T_1$ has the same generalized $\tau$-invariant as the Harish-Chandra module $X$ corresponding to $\sigma$, whence by \cite[Theorem 3.5.9]{G93} $I$ is indeed the annihilator of $X$, since primitive ideals of trivial infinitesimal character in type $C$ are uniquely determined by their generalized $\tau$-invariants.
  \end{proof}
  
  \noindent We also see that, since the wall-crossing operators $T_{\alpha\beta}$ generate the Harish-Chandra cells for Sp$(p,q)$ \cite[Theorem 1]{McG98}, modules in the same Harish-Chandra cell for this group (and trivial infinitesimal character) have the same signed tableaux $\T_2$ attached to them, up to changing the signs in double rows whose rows have even length.
  
  \section{$H$ is a bijection}
  
  \begin{theorem}\label{theorem:bijection} The map $H$ defines a bijection between $\mathcal S_{n,p}$ and ordered pairs $(\T_1,\overline{\T}_2)$, where $\T_1$ is a domino tableau with shape a doubled partition of $2(p+q)$ and $\overline{\T}_2$ is an equivalence class of signed tableaux of signature $(2p,2q)$ and the same shape as $\T_1$.
  \end{theorem}
  
  \begin{proof} We first show that any ordered pair $(\T_1,\overline{\T}_2)$ as in the hypothesis lies in the range of $H$, by induction on $p+q$.  Assuming that this holds for all pairs $(\T_1,\overline{\T}_2)$ if $\T_1$ has fewer than $n=p+q$ dominos, let $\T_1,\overline{\T}_2$ be a pair with $n$ dominos in $\T_1$.  Let $\T_1'$ be $\T_1$ with the $n$-domino removed.
  
If the $n$-domino in $\T_1$ is horizontal and lies in a row of even length, then the next to last row $R$ of $\T_1'$ has two more squares than its last row. By \cite[1.2.13]{G90}, there is a domino tableau $\T$ whose shape is that of $\T_1'$ with the last two squares removed from $R$ such that inserting a horizontal $i$-domino for a suitable index $i$ into the first row of $\T$ produces the tableau $\T_1'$, or else there is such a tableau $\T$ and an index $i$ such that inserting a vertical $i$-domino into the first column of $\T$ produces $\T_1'$.  In either case there is a pair $(\T_1'',\overline{\T}_2'')=H(\sigma')$ in the range of $H$, and if we add $(i,n)^+$ or $(i,n)^-$ to $\sigma'$ to get $\sigma$ (the first pair if the $i$-domino is horizontal, the second if it is vertical), then $H(\sigma) = (\T_1,\overline{\T}_2)$, as desired.  If instead the $n$-domino in $\T_1$ is horizontal but lies in a row of odd length then we can move $\T_1'$ through a suitable open cycle to produce a new tableau $\T_1''$ with shape a doubled partition such that the shape of $\T_1$ differs from that of $\T_1$ by a single vertical domino.  We then reduce to the case covered in the following paragraph.
 
Now suppose that the $n$-domino in $\T_1$ is vertical, so that the shape of $\T_1'$ is that of a doubled partition.  Let $\T_2$ be a representative of $\overline{\T}_2$; assume for definiteness that the squares in $\T_2$ corresponding to those of the $n$-domino in $\T_1$ are labelled $+$.  Look at all the double rows in $\T_2$ above the one corresponding to the double row with the $n$-domino in.$\T_1$.  If every such double row consists of rows of odd length ending in $+$, then one checks immediately that there is a class $\overline{\T}_2'$ such that the pair 
$(\T_1,\overline{\T}_2')=H(\sigma')$ lies in the range of $H$ and if we add $(n,+)$ to $\sigma'$ the resulting involution $\sigma$ satisfies $H(\sigma) = (\T_1,\overline{\T}_2)$ as desired.  Otherwise, if the lowest such double row $D$ has rows of odd length and ends in $-$, let $\tilde\T_1'$ be obtained from $\T_1'$ by removing the last squares of the rows of $D$.  There is a domino tableau $\T$ with the same shape as $\tilde\T_1'$ and an index $i$ such that inserting a suitably oriented $i$-domino into $\T$ gives $\T'$.  As above there is a class $\overline{\T}_2''$ of signed tableaux such that  $(\T,\overline{\T}_2'') = H(\sigma')$ and then there is $\sigma$ with $H(\sigma) = (\T_1,\overline{\T}_2)$, as desired.  If the lowest such double row has rows of even length ending in $+$, then look at the open cycle through the largest domino in the corresponding double row.   The argument of the last paragraph produces the desired $\sigma$.  If the lowest such double row $D$ has rows of even length ending in $-$, then look at the open cycle of $\T_1'$ through the largest domino in the corresponding double row.  If this open cycle has its hole and corner in different double rows $D_1,D_2$, then change all signs in these double rows of $\T_2$ and argue as in the previous case.  Finally, if this open cycle has its hole and corner  both in $D$, then move through this open cycle in $\T_1'$ and argue as in the case where the $n$-domino in $\T_1$ is horizontal.  In this case adjoining the $i$-domino initially produces a domino tableau where the first row of $D$ has length two more than its second row; moving through the open cycle, as specified by Definition 5.2 (2), gives $D$ the shape it has in $\T_1$ and then bumping the $n$-domino into the next lower double row yields $\T_1$, as desired.

Now we know that $H$ is surjective.  To show that it is injective, it is enough to show that its domain and range have the same cardinality.  To this end, we appeal to \cite{McG98}.  The cells of Harish-Chandra modules for Sp$(p,q)$ span complex vector spaces which carry the structure of representations of the Weyl group $W$ of type $C_{p+q}$.  Let ${\mathbf p}$ be a doubled partition of $2(p+q)$, with Lusztig symbol $s$, and let $\pi$ be the corresponding irreducible representation of $W$.  Enumerate the distinct even parts of $\mathbf p$ as $r_1,\ldots,r_k$ and denote by $\mathbf p_1,\ldots,\mathbf p_{2^k}$ the $2^k$ partitions obtained from $\mathbf p$ by either replacing the block $r_i\ldots,r_i$ of parts of $\mathbf p$ equal to $r_i$ by $r_i+1,r_i,\ldots,r_i,r_i-1$ or leaving this block unchanged, for all $i$ between 1 and $k$.  Then the $\mathbf p_i$ correspond (via their Lusztig symbols) to the representations in the complex double cell of $\pi$ of Springer type in the sense of \cite{McG98}.  From this and \cite[Corollary 3]{McG98} it follows that the number of equivalence classes ${\overline \T}_2$ of shape $\mathbf p$ relative to a fixed domino tableau $\T_1$ of this shape equals the number of modules in any Harish-Chandra cell $C$ with annihilator the primitive ideal corresponding to $\T_1$, provided that $C$ has at least one such module.  Hence the domain and range of $H$ have the same cardinality and $H$ is a bijection.
\end{proof}

\noindent  Fix a signed tableau $\T_2'$ whose rows of even length all begin with $+$.  It follows that signed involutions $\sigma$ such that the normalization (in the sense of the paragraph after Definition 5.2) of the second coordinate of $H(\sigma)$ is $\T_2'$ correspond bijectively to modules in a Harish-Chandra cell for Sp$(p,q)$ and that all such cells (of modules with trivial infinitesimal character) arise in this way; in particular, and in accordance with \cite[Theorem 6]{McG98}, there are as many such cells as there are nilpotent orbits in $\mathfrak g_0$.  It remains to show that all modules in the cell corresponding to $\T_2'$ have associated variety equal to the closure of the corresponding $K$-orbit in $\frak p$ via \cite[9.3.5]{CM93}.  This we will do in the next and final section.

\section{Associated varieties}

\noindent Our final result is

\begin{theorem}\label{theorem:associated varieties} Let $\sigma\in\mathcal S_{n,p}$ correspond to the Harish-Chandra module $Z$.  Then the associated variety of $Z$ is the closure of the $K$-orbit corresponding to $\T_2'$, where $H(\sigma)=(\T_1,\overline{\T}_2), \T_2$ is a representative of
$\overline{\T}_2$, and $\T_2'$ is its normalization as defined after Definition 5.2 (obtained from $\T_2$ by changing signs as necessary in all rows of even length so that they begin with $+$).
\end{theorem}

\begin{proof}
Let $\mathfrak q$ be a $\theta$-stable parabolic subalgebra of $\mathfrak g$ whose corresponding Levi subgroup of $G$ is Sp$(p',q')\times U(p_1,q_1)\times\cdots\times U(p_r,q_r)$, where the $p_i,q_i$ are such that $p'+\sum_i p_i = p, q" + \sum_i q_i = q$.  There is a simple derived functor module $A_{\mathfrak q}$ of trivial infinitesimal character whose associated variety is the closure of the Richardson orbit $\mathcal O$ attached to $\mathfrak q$ in the sense of \cite{T05}.  The corresponding clan $\sigma'$ is obtained as follows.  Its first block of terms corresponds to the factor Sp$(p',q')$, taking the form $(1,+)\ldots,(p'-q',+),(p'-q'+1,p'-q'+2)^-,\ldots,\hfil\linebreak(p'+q'-1,p'+q')^-$ if $p'\ge q'$ or $(1,-),\ldots,(q'-p',-),(q'-p'+1,q'-p'+2)^-,\ldots,\hfil\linebreak(q'+p'-1,q'+p')^-$ if $q'>p'$.  Its next block of terms takes the form\hfil\linebreak $(m+1,+),\ldots,(m+p_1-q_1,+),(m,m+p_1-q_1+1)^+,\ldots,(p'+q'+1,p'+q'+p_1+q_1)^+$, if $p_1\ge q_1$, or $(m+1,-),\ldots,(m+q_1-p_1)^-,(m,m+q_1-p_1+1)^+,\ldots\hfil\linebreak(p'+q'+1,p'+q'+p_1+q_1)^+$ if $q_1>p_1$ (where
$m+1 = \lfloor(p'+q'+p_1+q_1+1)/2\rfloor$); the remaining blocks of $\sigma$ correspond similarly to the remaining factors $U(p_i,q_i)$.  Letting $H(\sigma') = (\T_1,\oT_2)$ and defining $\T_2'$ as above, one checks immediately that the orbit corresponding to $\T_2'$  is indeed $\mathcal O$.  More generally, let $X'$ be any simple Harish-Chandra module for Sp$(p',q')$ with trivial infinitesimal character and associated variety $\bar{\mathcal O}$.  Then there is a simple Harish-Chandra module $X$ for $G$ obtained from $X'$ by cohomological parabolic induction from $\mathfrak q$, whose associated variety is the closure of the orbit induced from $\mathcal O$ in the sense of \cite{T05}.  Its signed involution $\sigma(X)$ is obtained from that of $X'$ by adding the blocks of terms corresponding to the $U(p_i,q_i)$ factors in the above construction of $\sigma$, and if the theorem holds for $X'$ and its associated variety, then the same is true for $X$.
  
Given $Z$ as in the theorem, let $\bar{\mathcal O}$ be its associated variety.  If $\mathcal O$ is the closure of a Richardson orbit, say the one attached to the $\theta$-stable parabolic subalgebra $\mathfrak q$, then the module $A_{\mathfrak q}$ above lies in the same Harish-Chandra cell as $Z$ and the theorem holds for $A_{\mathfrak q}$, whence it holds for $Y$.  In general, using \cite[Proposition 2.3 (3)]{T05} and induction by stages, we can induce $\mathcal O$ to an orbit $\mathcal O'$ for a higher rank group $G'$ such that all even parts in the partition corresponding to $\mathcal O'$ have multiplicity at most $4$, whence $\mathcal O'$ is Richardson by \cite[Corollary 5.2]{T05}.  Then the result holds for the module $Z'$ correspondingly induced from $Z$.  But now the orbit $\mathcal O$ of Sp$(p,q)$ is the only one inducing to $\mathcal O'$ relative to a suitable $\theta$-stable parabolic subalgebra $\mathfrak q$ of Lie~$G'$ with Levi subgroup having Sp$(p,q)$ as its only factor of type $C$.  It follows that the theorem holds for $Z$, as desired.
\end{proof}

\end{document}